\documentclass[11pt]{amsart}

\usepackage{amsmath,amsfonts,amssymb,latexsym}

\usepackage{amsmath}
\usepackage{amssymb}
\usepackage{amsthm}    
\usepackage{amscd}     
\usepackage[all]{xy}  
\usepackage{remreset}   
\usepackage{latexsym}     

\newcommand\NN{\mathbb N}
\newcommand\ZZ{\mathbb Z}

\newcommand\RR{\mathbb R}
\newcommand\CC{\mathbb C}

\newcommand\bcp{\mathbb C \mathbb P}
\newcommand\cpb{\overline{\mathbb C \mathbb P^2}}

\newcommand\wt{\widetilde}

\theoremstyle{plain}

\newtheorem{thm}[equation]{Theorem}

\newtheorem{cor}[equation]{Corollary}

\newtheorem{prop}[equation]{Proposition}

\newtheorem{rmk}[equation]{Remark}

\theoremstyle{definition}

\newtheorem{main}{Theorem}

\newtheorem*{ack}{Acknowledgments}

\makeatother

\begin{document}
\sloppy

\title{On Finite Symmetries of Simply Connected Four-Manifolds}
\author{Ioana {\c S}uvaina}
\thanks{Supported in part by NSF grant DMS-1309029.}

\address{Ioana {\c S}uvaina, 
         Department of Mathematics, Vanderbilt University, 1326 Stevenson Center, Nashville, TN 37214}
        
\email{ioana.suvaina@vanderbilt.edu}
\date{\today}

\begin{abstract} For most positive integer pairs $(a,b)$,
 the topological space $\#a\bcp^2\#b \cpb$ is shown to 
admit infinitely many inequivalent smooth structures
which dissolve upon performing  a single connected sum 
with $S^2\times S^2$.  This is then used  to 
construct  infinitely many non-equivalent smooth free actions of suitable finite groups on the connected sum 
 $\#a\bcp^2\#b \cpb$. We then investigate  the behavior of the sign of  the Yamabe invariant for the resulting finite covers, and observe  that 
these constructions provide many new counter-examples to the $4$-dimensional  Rosenberg Conjecture.
\end{abstract}
\maketitle

\section{Introduction} \label{introduction}

The geometry of  manifolds in dimension $4$ is remarkably intricate, and    in key  respects differs markedly
 from the  corresponding story in any other dimension.
For example,  many compact topological $4$-manifolds can be shown to support infinitely many inequivalent smooth structures.
Indeed, the remarkable results of Freedman \cite{freedman} and Donaldson \cite{don}  show that oriented 
simply-connected compact $4$-manifolds without boundary are classified, up to  homeomorphism, 
by just three invariants: the Euler characteristic $\chi$, the signature $\tau$, 
and the parity of the intersection form on $H^2({\mathbb Z})$; consequently, any simply connected 
non-spin closed $4$-manifold  is homeomorphic to a connected sum $\#a\bcp^2\#b\cpb$, 
where  $\bcp^2$ and $\cpb$  respectively denote the complex projective plane  with its standard and non-standard orientations. 
By contrast, however, gauge theory gives rise to diffeomorphism invariants, such as 
 Seiberg-Witten basic classes,  that can often be used to show that pairs of 
   $4$-manifolds 
are 
 non-diffeomorphic, even though the above-mentioned results imply that they are actually homeomorphic. 
 If a  $4$-manifold $M$ is orientedly diffeomorphic to $\#a\bcp^2\#b\cpb$, equipped with the  familiar  smooth structure arising from the 
 standard smooth structure on the complex projective plane via the connect-sum construction, 
 we
 will therefore  find it useful to say that $M$ carries the {\em conventional} smooth structure; on the other hand, if $M$ is homeomorphic but
 not diffeomorphic to the connected sum $\#a\bcp^2\#b\cpb$, we will then say that its differential structure is {\it exotic}.

We first discuss the existence of exotic structures and the region in which they can be exhibited,
and emphasize a special property of these structures, namely their solubility.
We say that a non-spin $4$-manifold is
{\it$(S^2\times S^2)$-soluble} 
or {\it $\bcp^2$-soluble,} if after taking the connected sum with
one copy of $S^2\times S^2$ or with one copy of $\bcp^2,$ respectively,
its  differential structure is the conventional structure. 
The $(S^2\times S^2)$-solubility is more adequate 
if one studies the general class of $4$-manifolds, 
while the $\bcp^2$-solubility is more suitable to the study of 
complex surfaces or symplectic $4$-manifolds. 
The $\bcp^2$-solubility property is equivalent to the almost complete decomposability
in the sense of Mandelbaum \cite{man}. 

We study the geography of  exotic structures on simply connected non-spin
$(S^2\times S^2)$-soluble
 manifolds. This generalizes the results on the geography of
symplectic manifolds of \cite{bra-kot}.
\begin{main}\label{geogr}
For any $\epsilon>0$  there is a constant $N_\epsilon>0$ 
such that given any integer pair $(a,b)$ in the
first quadrant satisfying either one of the two conditions
\begin{eqnarray}
b\geq(\frac12+\epsilon)a+N_\epsilon ~and~ a \not\equiv 0 \mod8, \label{SW-nontr}\\
b\leq\frac2{1+2\epsilon}(a-N_\epsilon) ~and~ b \not\equiv 0 \mod8, \label{op-orient}
\end{eqnarray}
the topological space $M=\#a\bcp^2\#b\cpb$ admits infinitely many pairwise non-diffeomorphic 
smooth structures, which are all
$(S^2\times S^2)$-soluble.

\end{main}

This result is used to exhibit an infinite 
 number of inequivalent smooth free actions on the {\it conventional}
non-spin smooth structure on $4$-manifolds:

\begin{main}
\label{main} Let $d\geq2$ be an integer and $\Gamma$ any group of order $d$
which acts freely on the sphere $S^3.$
For any $\epsilon>0$ there exists a constant $N_\epsilon'>0$
such that for any  point $(a,b)$ in the region $R_\epsilon$
satisfying the divisibility conditions $(D_1)$ and $(D_2)$
the manifold $M=\#a\bcp^2\#b\cpb$ has the following properties:
 \begin{itemize}
\item[${\ref{main}}_1$:]  $M$ admits infinitely many smooth orientation preserving free actions of the group $\Gamma,$ 
which we denote by $\Gamma_i, i\in \NN,$
\item[${\ref{main}}_2$:] the actions $\Gamma_i$ 
are conjugate by homemorphisms, but are not conjugate by the diffeomorphisms of $M.$
 \end{itemize}

The region $R_\epsilon$ is defined as:
\begin{eqnarray*}
R_{\epsilon,1}&=&\{(a,b)\in\NN\times\NN~|~b\geq(\frac12+\epsilon)a+N_\epsilon'\} \\
R_{\epsilon,2}&=&\{(a,b)\in\NN\times\NN~|~b\leq\frac2{1+2\epsilon}(a-N_\epsilon') \}\\
R_{\epsilon}&=&R_{\epsilon,1}\cup R_{\epsilon,2}
\end{eqnarray*}
while  the divisibility conditions are:
\begin{eqnarray*}
D_1:&a+1\equiv0\mod d {\text{ and }} b+1\equiv0\mod d,\\
D_2:& {\text{ if $(a,b)\in R_{\epsilon,1}$ then $\frac{a+1}d \not\equiv 1 \mod8$}},\\
&{\text{or if $(a,b)\in R_{\epsilon,2}$ then $\frac{b+1}d \not\equiv 1 \mod8$}}.
\end{eqnarray*}
 \end{main}

For any integer $d$ there exists at least one group of order $d$ which acts freely on the
$3$-sphere, for example the finite cyclic group $\ZZ_d\cong\{\rho\in\CC|~ \rho^d=1\}$ acting on 
$S^3\subset\CC^2$ by multiplication.
The divisibility condition $D_1$ is necessary in order to assure that the quotient manifolds 
$M/\Gamma_i$ have the Betti numbers $b_2^+=\frac{1+a}d-1$ and $b_2^-=\frac{1+b}d-1$ integer valued.
The constant $\epsilon$ can be chosen to be arbitrary small. 
Unfortunately, due to the nature of the constructions involved we are unable to compute the value 
of $N_\epsilon'.$ The constant depends on the constant $N_\epsilon$ from Theorem \ref{geogr} and
increases linearly in $d,$ for the exact formula see equation (\ref{constant}) in  Section $\S$ \ref{exotic-str}.
 Nevertheless, the region $R_\epsilon$ covers all the integer lattice points 
in the first quadrant with the exception of finitely many.

This theorem generalizes earlier work of Ue \cite[Main Theorem]{ue}, 
LeBrun \cite[Theorem 2]{lnyj} and Hanke-Kotschick-Wehrheim \cite[Theorem 12]{hkw},
putting an emphasis on the geography of the $4$-manifolds. 
 Torres \cite[Theorems 1.6 and 1.7]{tor} also discusses similar results but with more restrictive 
divisibility conditions and for a much smaller region.  
Using different constructions, Akhmedov-Ishida-Park \cite{aip} exhibited a similar
behaviour for manifolds with zero signature. 
One of the aims of these papers is to provide examples for which the 
Rosenberg Conjecture \cite[Conjecture 1.2]{rosen} fails, 
meaning that there are manifolds with finite fundamental group of odd order  
which do not admit metrics of positive scalar curvature, while their universal cover does.
Such a phenomenon can be detected by a differential invariant of the manifold arising from 
Riemannian  geometry, the Yamabe invariant, see equation (\ref{yam}) in Section $\S$ \ref{prop-yam}.
As $\bcp^2$ admits a metric of positive curvature, for example the Fubini-Study metric, 
then a result of Gromov-Lawson \cite{gr-la} tells us that this property is preserved 
under the  connected sum operation,
and hence so does  $M=\#a\bcp^2\#b\cpb.$ 
The  Yamabe invariant of $M$ is, then, positive. 
The existence of a symmetry group on  $M$  has an immediate impact on the
Riemannian properties of the manifold.
If we consider the manifold $M$ endowed with one of the actions $\Gamma_i$,
 we  show that:
\begin{prop}\label{yamabe}
For any integer $d$ and any point $(a,b)$ in the region $R_\epsilon$ satisfying 
the divisibility conditions $D_1$ and $D_2$
let $g$ be a $\Gamma_i$-invariant metric on 
the manifold $M=\#a\bcp^2\#b\cpb$, for any of the $\Gamma_i$ actions defined in Theorem \ref{main}.
Then the Yamabe invariant of the conformal class of $g$ is negative.
\end{prop}

\begin{rmk}
In some cases it can be shown that these actions also give obstructions
 to the existence of invariant Einstein metrics, see
\c Suvaina \cite{eu-einst}.
\end{rmk}

For the examples in Proposition \ref{yamabe} the Yamabe invariant of the conformal class of 
a $\Gamma_i$-invariant metric, or equivalently of any conformal class on $M/\Gamma_i,$ is negative.
Thus our constructions provide a large class of counterexamples to the Rosenberg Conjecture.
Moreover, the Yamabe invariant of the conformal class is bounded away from zero.
This is not necessary true  for an arbitrary action of $\Gamma$ and $\Gamma$-invariant metrics.

\begin{prop}\label{yam0}
For any integer $d\geq2$ and any integer pair $(a,b)$ in the region:
\begin{eqnarray*}
R_0 =&\{(a,b)\in\NN\times\NN~|~b\geq5a+12d+4, {\text{ and }}\frac{1+a}d \not\equiv 1 \mod8\} \\
&\bigcup~\{(a,b)\in\NN\times\NN~|~b\leq\frac15(a-12d-4), {\text{ and }}\frac{1+b}d \not\equiv 1 \mod8\}
\end{eqnarray*}
on the manifold $M=\#a\bcp^2\#b\cpb$ there exist infinitely many conjugate homeomorphic 
conjugate non diffeomoprhic actions of the group $\Gamma$, which we denote by $\Gamma_i', i\in\NN,$ such that the
Yamabe invariant of the conformal class of a $\Gamma_i'$-invariant metric $g$ is nonpositive  
and for different choices of $\Gamma_i'$-invariant  metrics $g$ the invariant   can
be made arbitrary close to zero. 
\end{prop}

As an immediate consequence of Propositions \ref{yamabe} and \ref{yam0} we have:
\begin{cor}\label{dist-yam}
On any manifold $M=\#a\bcp^2\#b\cpb$ for $(a,b)$ in the region $R_0\cap R_\epsilon$
 there are infinitely many, free actions of the group $\Gamma,$ denoted by $\Gamma_i, i\in\NN,$
such that the Yamabe invariant of the  manifold $M/\Gamma_i$ is negative, and
there are also infinitely many free actions of the group $\Gamma,$ denoted by $\Gamma_i', i\in\NN,$
such that the Yamabe invariant of the manifold $M/\Gamma_i'$ is zero.
\end{cor}


\section{Exotic smooth structures}\label{exotic-str}
In this section we discuss the existence of exotic smooth structures and prove Theorem \ref{geogr}.
\begin{proof}[Proof of Theorem \ref{geogr}]
It is enough to prove the statement for the points $(a,b)$ satisfying the  inequality (\ref{SW-nontr}).
The second inequality is immediately implied by the first if one considers a change in orientation.
The solubility property is  preserved under this operation as $S^2\times S^2$ 
admits a diffeomorphism which reverses its orientation,
for example one can consider the antipodal map on one of the factors and identity on the other.

The construction relies on a result of Braungardt-Kotschick, which proves the following theorem 
about the geography of
$\bcp^2$-soluble symplectic manifolds:
\begin{thm}[ {\cite[Theorem 4]{bra-kot}}]\label{th-BK}
For every $\epsilon>0$ there is a constant $N_\epsilon>0$ such that every lattice point $(a,b)$ 
satisfying the conditions
\begin{eqnarray}
a\equiv 1 \mod2\label{odd}\\
b\leq4+5a\label{minimal}\\
b\geq(\frac12+\epsilon)a+N_\epsilon \label{B_K-ineq}
\end{eqnarray}
is realized by the Betti two invariants $(a,b)=(b^+_2,b^-_2)$ of infinitely many
pairwise non-diffeomorphic  simply connected minimal symplectic manifolds $M_{(a,b,i)}$, all of which are 
$\bcp^2$-soluble.
\end{thm}

As the manifolds considered are simply connected and symplectic then $a=b_2^+$ must be odd, namely 
condition (\ref{odd}), 
and  furthermore condition (\ref{B_K-ineq}) implies that $a$ is a  large integer. 
For minimal symplectic manifolds with $b^+>1,$ Taubes \cite{tau} showed that 
 $c^2_1=4+5b^+_2-b^-_2\geq0$, which is equivalent to
 inequality (\ref{minimal}).
In the original paper, the geography statement is given in terms of the Chern number $c_1^2,$ 
and the Todd number $\chi_h.$ 
These can be computed as $c_1^2=2\chi+3\tau=2+5b_2^+-b_2^-$ and 
$\chi_h=\frac14(\chi+\tau)=\frac12(1+b^+_2).$ 
After relabeling the constants $\epsilon$ and $N_\epsilon,$ 
the Kotschick-Braungardt inequality in \cite{bra-kot} can be formulated as inequality (\ref{B_K-ineq}).
The infinitely many smooth structures in the theorem are constructed by doing logarithmic transforms of 
different multiplicities along a symplectic $2$-torus of self-intersection zero.
To prove that this construction generates infinitely many smooth structures, 
one  considers a smooth invariant called the bandwidth 
$\mathcal{BW}$. 
This is defined to be the highest divisibility of the difference between two distinct Seiberg-Witten basic classes, 
see Ishida-LeBrun \cite[Definition 2]{i-lb}.
As the number of basic classes
of a manifold is finite while 
 $\displaystyle \sup_{i\in\NN} {\mathcal BW}(M_{(a,b,i)})=+\infty,$ we conclude that  infinitely
many diffeotypes are represented.

The blow-up of the manifold $M_{(a,b,i)}$ at $p$ points 
 is a symplectic manifold  diffeomorphic to $M_{(a,b,i)}\#p\cpb.$ Moreover, it is 
  simply connected and $\bcp^2$-soluble.
A trick of Wall \cite{wall} tells us that if $X$ is a non-spin manifold then $X\#(S^2\times S^2)$
 is diffeomorphic to $X\#(\bcp^2\#\cpb).$ 
 Hence, $M_{(a,b,i)}\#p\cpb, p\geq0,$ is also $(S^2\times S^2)$-soluble.

We denote by $ E_j$ the exceptional $2$-spheres of self-intersection $(-1)$ introduced by the blow-ups, 
and by $e_j=c_1(E_j)$ the Poincare dual of the homology class of $E_j.$
If $s$ is a basic class of $M_{(a,b,i)},$ then 
$\pm s+\sum_{j=1}^p \pm e_j$ are basic classes of $M_{(a,b,i)}\#p\cpb.$
Hence ${\mathcal BW}(M_{(a,b,i)}\#p\cpb)\geq {\mathcal BW}(M_{(a,b,i)}),$
and the bandwidth argument implies that the 
family $M_{(a,b,i)}\#p\cpb$ contains infinitely many diffeotypes.
In particular, we  showed that for each integer pair $(a,b)$ satisfying condition
(\ref{odd}) and condition (\ref{B_K-ineq}) there are infinitely many
pairwise non-diffeomorphic simply connected $S^2\times S^2$-soluble
$4$-manifolds with $(b_2^+,b_2^-)=(a,b).$

In order to generalize condition (\ref{odd}) in Theorem \ref{th-BK} we use the stable cohomotopy 
Seiberg-Witten invariant, due to Bauer-Furuta \cite{baufu,bauer2}.
In \cite[Proposition 4.5]{bauer2}, Bauer  considers manifolds obtained by
taking the connected sum of $k=2, 3$ or $4$ manifolds $\displaystyle X=X_1\#\cdots\# X_k$
 satisfying $b_1(X_j)=0, b_2^+(X_j)\equiv 3\mod4,$ and if $k=4 \sum_1^4b_2^+(X_j)\equiv4\mod8.$
 If the Seiberg-Witten invariant of a basic class $s_j$ of $X_j$ is odd, 
then $X$ has a non-vanishing  Bauer-Furuta invariant for  basic classes
 of the form 
$$\pm s_1 \pm \dots \pm s_k, {\text{ for }} k=2,3 {\text{ or }} 4.$$

Let $X_{1,i}=M_{(a,b,i)}\#p\cpb$ and $X_2=X_3=X_4=K3,$ the simply connected complex surface with trivial canonical line bundle. 
The $K3$ surface can be realized, for example, as a hypersurface of degree $4$ 
in $\bcp^3,$ and has $b_2^+=3$ and $b_2^-=19.$
Consider the following three families: 
$$X_{1,i},~X_{2,i}=X_{1,i}\#X_2,$$
$$X_{3,i}=X_{1,i}\#X_2\#X_3\#X_4,   {\text{ only when }} b_2^+(X_{1,i})+3\cdot3\equiv4\mod8.$$
As  $X_{1,i}, X_2, X_3, X_4$ are symplectic manifolds with $b_2^+>1,$  Taubes \cite{taubes} showed that
the Seiberg-Witten invariant associated
to a canonical almost complex structure is $\pm1$.
Hence we can apply Bauer's theorem and the bandwidth argument to conclude that infinitely many
smooth structures are represented. These three families cover all the points in the region defined by
inequality (\ref{SW-nontr}), up to a change of the constant $N_\epsilon.$

The $K3$ manifold is $\bcp^2$-soluble \cite{man}. 
 As $X_{1,i}$ is $(S^2\times S^2)$-soluble, for $k=2,$
 we can now conclude  that:
 \begin{align}
 X_{2,i}\#(S^2\times S^2)&\cong X_{1,i}\#(S^2\times S^2)\#K3\cong (a+1)\bcp^2\#(b+p+1)\cpb\#K3\notag\\
 &\cong (a+1+3)\bcp^2\#(b+p+1+19)\cpb.\notag
 \end{align}
 
A similar computation for $k=4$ shows that $X_{3,i}$ is $S^2\times S^2$-soluble.
This  concludes the proof of Theorem \ref{geogr}.
\end{proof}

\section{Infinitely many free actions}\label{actions}

In this section we use the examples in Theorem \ref{geogr} to construct infinitely many actions of
finite groups on manifolds with the conventional smooth structure and prove Theorem \ref{main}.
\begin{proof}[Proof of Theorem \ref{main}]
We only need to prove that the statement is true for the points in the first subset, 
as considering  the opposite orientation will imply the result for the second subset.

For any finite group $\Gamma$ acting freely on $S^3$ consider  the orientable rational homology sphere 
$S_\Gamma$ with fundamental group $\pi _1(S_\Gamma)=\Gamma$ and universal cover
 $\#(d-1) (S^2\times S^2),$ constructed by Ue \cite[Proposition 1-3]{ue}. 
Corollary 8 in \cite{i-lb-o} shows that the Bauer-Furuta invariant of $X_{k,i}\#S_\Gamma$ is non-trivial for
 monopole classes of the form $s_i\in H^2(X_{k,i},\ZZ)\hookrightarrow H^2(X_{k,i}\#S_\Gamma,\ZZ)/{torsion},$ where $s_i$ is a basic class for
 the Bauer-Furuta invariant on $X_{k,i},$ for $k=1, 2$ or $3.$
The bandwidth argument can be used again  to argue that we constructed infinitely many diffeotypes.

Moreover, the universal cover of $X_{k,i}\#S_\Gamma$ is diffeomorphic to
$$\#d X_{k,i}\#(d-1)(S^2\times S^2)\cong [d\cdot b_2^+(X_{k,i})+d-1]\bcp^2\#[d\cdot b_2^-(X_{k,i})+d-1]\cpb.$$
Hence, on the conventional smooth structure of 
$\#a\bcp^2\#b\cpb=[d\cdot b_2^+(X_{k,i})+d-1]\bcp^2\#[d\cdot b_2^-(X_{k,i})+d-1]\cpb$ 
we constructed infinitely many free actions of the group $\Gamma$ such that the quotient spaces
are homeomorphic but pairwise non-diffeomorphic. 

It is easy to check that these constructions cover all the points in the  region 
$R_\epsilon$ satisfying conditions $D_1$ and $D_2$ for 
\begin{equation}\label{constant}
N_\epsilon'=dN_\epsilon+(\frac12-\epsilon)(d-1),
\end{equation}
as we consider all the manifolds $X_{k,i}$ used  in the proof of Theorem \ref{geogr}.
\end{proof}

\section{The sign of the Yamabe invariant}\label{prop-yam}
We consider next
an invariant associated to Riemannian metrics.
 On a closed $4$-manifold $M$, given a Riemannian metric $g, $ one defines the Yamabe invariant of the conformal class $[g]$ of $g$  as:
$$Y(M,[g])=\inf_{\wt g\in[g]}\frac{\int_M s_{\wt g} d\mu_{\wt g}}{{Vol(\wt g)}^{\frac12}},$$
where $[g]=\{\wt g=e^fg|~ f:M\to\RR \text{ smooth}\},$ see for example \cite{schoen, lepa}.

The Yamabe invariant is a diffeomorphism invariant of the manifold and is defined  \cite{oKob,schoen-y} as:
\begin{equation}
\label{yam}
Y(M)=\sup_{[g]} Y(M,[g]).
\end{equation}

\begin{proof}[Proof of Proposition \ref{yamabe}]
Given one of the $\Gamma_i$ actions on $M$ constructed in Theorem \ref{main}  the
quotient $M/\Gamma_i$ is diffeomorphic to $X_{k,i}\#S_\Gamma,$ with the previous notations.
As the manifold $X_{k,i}\#S_\Gamma$ has a non-trivial Bauer-Furuta invariant, 
then $Y(X_{k,i}\#S_\Gamma)\leq0$ \cite{i-lb-o}.
Moreover, by  \cite[Proposition 12 and 14]{i-lb-o}  the Yamabe invariant satisfies:
$$Y(X_{k,i}\#S_\Gamma)\leq-4\pi\sqrt{2c_1^2(M_{(a,b,i)})+2\sum_{j=2}^kc_1^2(X_j)}
=-4\pi\sqrt{2c_1^2(M_{(a,b,i)})},$$
where $M_{(a,b,i)}$ denotes the minimal symplectic manifold used in the proof of 
Theorem \ref{geogr}. Moreover, we can assume that $M_{(a,b,i)}$ is a symplectic manifold of general type, 
meaning that $c_1^2(M_{(a,b,i)})>0.$ 
This implies that $Y(X_{k,i}\#S_\Gamma)\leq-4\pi\sqrt2$ and, in particular, that there is a
metric of negative constant scalar curvature in the conformal class of $\pi_*(g),$ 
where $\pi_*$ denotes the push
forward of the $\Gamma_i$-invariant metric $g.$ 
Then the conformal class $[g]$ contains a negative constant
scalar curvature metric (which is $\Gamma_i$-invariant) and hence, its Yamabe invariant is negative
\cite{schoen}. This concludes the proof of the proposition.
\end{proof}

In order to prove  Proposition \ref{yam0} we need a different construction of group actions.
\begin{proof}[Proof of Proposition \ref{yam0}]
The main blocks in the construction are  the elliptic surfaces. 
As they are well understood from the algebraic geometry, 
differential topology, or Seiberg-Witten theory point of view, we give the reader a textbook  reference \cite[Chapter 3]{go-st},
where all these aspects are presented and complete references are included. 
We denote by $E(n)$ the simply connected elliptic surface with Euler characteristic $\chi(E(n))=12n$ and
signature $\tau(E(n))=-8n$ which admits a section. Let $E(n)_{p,q}$ the complex surface obtained by doing 
two logarithmic transforms of multiplicities $p$ and $q$ on two generic fibers and require that  $(p,q)=1,$ in
order for the manifold to be simply connected. 
It is well known that $E(n)_{p,q}$ is diffeomorphic to $E(n)_{p',q'}$ if and only if $\{p,q\}=\{p',q'\}$ for $n\geq2,$
\cite[Corollary 3.3.7]{go-st}
and when $n=1$ if $\{p,q\}=\{p',q'\}$ or if $1\in\{p,q\}\cap\{p',q'\}$ \cite[Theorem 3.3.8]{go-st}.
Moreover, $E(n)_{p,q}$ is spin if and only if $n$ is even and $pq$ is odd \cite[Lemma 3.3.4]{go-st}.
In particular, if $n$ is odd or if $n$ is even and $pq$ is even then the manifolds $E(n)_{p,q}$ are all non-spin
and homeomorphic, for a fixed $n.$ Moreover, in this case $E(n)_{p,q}$ is 
$\bcp^2$-soluble \cite[Theorem 2.15]{man}.
Using Wall's trick \cite{wall}, then this is also $S^2\times S^2$-soluble.

As before, $\Gamma$ is a finite group of order $d$ acting freely on $S^3$ 
and $S_\Gamma$ is Ue's  rational homology sphere with fundamental group $\Gamma.$
We consider the following three families of manifolds:

\begin{align*}
\mathcal F_1(n)=&\{E(n)_{p,q}\#k\cpb\#S_\Gamma| k\in\NN, (p,q)=1, pq {\text{ even and if }} n=1 {\text{ then }}p,q\geq2\}\\
\mathcal F_2(n)=&\{E(n)_{p,q}\#k\cpb\#S_\Gamma\#E(2)| k\in\NN, (p,q)=1, pq {\text{ even and }} n  {\text{ even}}\}\\
\mathcal F_3(n)=&\{E(n)_{p,q}\#k\cpb\#S_\Gamma\#3E(2)| k\in\NN, (p,q)=1,  pq{\text{ even and }} n \equiv 2\mod 4\}
\end{align*}

For fixed $n$ and $k$ the manifolds in each family are homeomorphic, moreover by using the bandwidth arguments due to \cite{i-lb} we  conclude that infinitely many diffeotypes are constructed. 
Their universal cover is  the conventional non-spin manifold.
The topological invariants of the first family are of the form
$$(a,b)=(b_2^+,b_2^-)=(2n-1,10n-1+k), n\geq1, k\geq0,$$
covering the region $$b\geq 5a+4 {\text{ for $a$ odd}}.$$
Hence, their universal coverings cover all the integer lattice points in the region:
$$(\frac{1+b}d-1)\geq 5(\frac{1+a}d-1)+4 {\text{ or }}b\geq 5a+4 {\text{ for $\frac{1+a}d$ even}}.$$
A similar computation for the second and third families shows that we are able to cover all the 
integer lattice points on the regions:
$$ b\geq 5a+4+4d{\text{ when }} n=\frac12(\frac{1+a}d-3) {\text{ even, or equivalently }}
 \frac{1+a}d\equiv3\mod4,$$
$$b\geq 5a+4+12d{\text{ when }} n=\frac12(\frac{1+a}d-9)\equiv 2\mod4 {\text{, or  }} \frac{1+a}d\equiv5\mod8.$$
The union of these sets covers the needed region.

The Yamabe invariant of the manifolds in the families $\mathcal F_{1,2,3}$ is zero, 
by \cite[Theorem A]{i-lb-o}. Hence, we can always choose a family of 
negative constant scalar curvature metrics 
for which the scalar curvature converges to zero. 
These lift to metrics on the universal covering such that the Yamabe invariants 
of their conformal classes are negative 
and converge to zero.
\end{proof}

\begin{proof}[Proof Corollary \ref{dist-yam}]
Propositions \ref{yamabe} and \ref{yam0} provide the constructions for the infinitely many actions for which the quotients have the Yamabe invariant negative or zero, respectively.
The quotient manifolds are all homeomorphic 
by the Donaldson and Freedman's classification \cite{don,freedman}, 
as they are all of the form $M_*\# S_\Gamma,$ with 
the manifolds $M_*$ being simply connected non-spin with the same topological invariants 
$b_2^+$ and $b_2^-.$
They are pairwise non-diffeomorphic, as the quotients in different families have different Yamabe invariants 
while in the same family they are non-diffeomorphic by construction.
\end{proof}

\begin{rmk}
If the Yamabe invariant of $S_\Gamma$ is positive, then we can also construct an action on $M$
such that $M/\Gamma$ is diffeomorphic to $\#m\bcp^2\# n\cpb\#S_\Gamma$ and 
has positive Yamabe invariant. This is true, if for example $\Gamma=\ZZ_2$ and
$S_{\ZZ_2}$  is the quotient of $S^2\times S^2$  by the diagonal antipodal map on each factor.
It is difficult to exhibit two homeomorphic, non diffeomorphic structures with positive Yamabe invariants
as  there are no known invariants to distinguish  such smooth structures. 
\end{rmk}

\begin{ack}
The author would like to thank Claude LeBrun for valuable comments which
 helped improve the presentation.
\end{ack}

\bibliographystyle{alpha} 

  \end{document}